\documentclass[11pt, reqno]{amsart}
\usepackage{amssymb, hyperref}
\usepackage{amsthm}
\usepackage{tikz}
\usepackage[all,2cell]{xy}
\usepackage{verbatim}
\usepackage{xcolor,soul}
\usepackage{hyperref}
\usepackage{mathrsfs}
\usepackage{csquotes}
\usepackage{epigraph}
\numberwithin{equation}{section}
\newcommand{\width}{{\mathrm{width}}}
\newtheorem{theorem}{Theorem}[section]
\newtheorem{corollary}[theorem]{Corollary}
\newtheorem{definition}[theorem]{Definition}
\newtheorem{lemma}[theorem]{Lemma}

\newtheorem{proposition}[theorem]{Proposition}

\theoremstyle{remark}
\newtheorem{remark}[theorem]{Remark}
\newtheorem*{theorem*}{Theorem}
\newtheorem*{corollary*}{\bf Corollary}

\theoremstyle{remark}
\newtheorem{example}[theorem]{Example}

\title[On the Gromov width of toric manifolds]{A note on the Gromov width of toric manifolds}

\author[B.~N.~Chary]{Narasimha Chary Bonala}
\address{Narasimha Chary Bonala\\ Ruhr-Universit\"at Bochum, Fakult\"at f\"ur Mathematik, D-44780 Bochum, Germany}

\curraddr{Department of Mathematics and Statistics, Indian Institute of Technology Kanpur,
U.P. India, 208016.
}
\email{Narasimha.Bonala@rub.de and chary@iitk.ac.in}

\author[S.~Cupit-Foutou]{St\'ephanie Cupit-Foutou}
\address{St\'ephanie Cupit-Foutou\\ Ruhr-Universit\"at Bochum, Fakult\"at f\"ur Mathematik, D-44780 Bochum, Germany}
\email{Stephanie.cupit@rub.de}

\thanks{This research was supported by the  CRC/TRR 191 “Symplectic Structures in Geometry, Algebra and Dynamics” of the Deutsche Forschungsgemeinschaft.}


\begin{document}

\begin{abstract}
The Gromov width of a uniruled projective K\"ahler manifold can be bounded from above by the symplectic area
of its minimal curves. We apply this result to toric varieties and thus get in this case upper bounds expressed in toric combinatorial invariants. 
\end{abstract}

\maketitle

\section{Introduction}
The Gromov width of a $2n$-dimensional symplectic manifold $(X, \omega)$ is defined as
$$w_G(X, \omega) := sup\{ a: (B^{2n}(\sqrt{a/\pi}),\omega_{st})~\text{symplectically embeds into}~ (X,\omega)\},$$
where $(B^{2n}(r), \omega_{st})$ is the ball of radius $r$ centered at the origin in $\mathbb R^{2n}$ and equipped with the standard symplectic form.
This is a symplectic invariant which is in general difficult to compute.
Computations and estimates of the Gromov width in various cases have been obtained by several authors (see for example~\cite{biran, kartol,LMZ,Cas,FLP,HLS} and references therein).

In this article, we consider projective K\"ahler manifolds which are uniruled, i.e. covered by rational curves.
As shown in the following theorem, 
the Gromov width of any uniruled projective K\"ahler manifold is bounded from above by the symplectic area of any minimal curve.

\begin{theorem}\label{thm:main1}
Let $X$ be a projective complex manifold and $\omega$ a K\"ahler form of $X$.
For any minimal curve $C$ of $X$, we have the inequality
$$
w_G(X, \omega)\leq  \int_C \omega.
$$
\end{theorem}
Note that minimal curves exist on uniruled projective complex manifolds (see e.g. Theorem~\ref{thm:existence of minimal}). 
Similarly as in some previous works (see e.g.~\cite{Lu06b,Cas}), to show Theorem~\ref{thm:main1}, 
one may apply the methods of Gromov's used in~\cite{Gro}.
This leads to prove the non-vanishing of some Gromov-Witten invariants. Here we concentrate only on minimal curves of projective K\"ahler manifolds, which enables to use some ideas of Koll\'ar and Ruan's to show the existence of such invariants; see Section~\ref{sec:upperbound} for details. 

To the best of our knowledge, the upper bound given in Theorem~\ref{thm:main1} is the best possible bound that can be obtained by Gromov-Witten invariants.
This bound is sharp for polarized coadjoint orbits, more generally for polarized Bott-Samelson varieties, and for generalized Bott manifolds as well. This follows from the main result of 
\cite{Cas} and \cite{FLP} together, \cite{BCF} and \cite{HLS} respectively, combined with the characterization of minimal curves of these varieties (see~\cite{BCF}). 

While specifying our results to toric varieties,
we obtain upper bounds for Gromov widths in terms of toric combinatorial invariants.
More precisely, let $(X,\omega)$ be a compact K\"ahler toric manifold.
Let $\Sigma$ denote the fan of $X$ and $\Sigma(1)$ be the set of one-dimensional cones in $\Sigma$. 
Then $[\omega]\in H^2(X,\mathbb R)$ can be written as 
\begin{equation} \label{eq:kaehlerclass}
[\omega]={\textstyle\sum}_{\rho\in \Sigma(1)}\kappa_{\rho}[D_{\rho}]
\end{equation}
where $[D_{\rho}]$ is the divisor class of $X$ associated to $\rho$ and $\kappa_{\rho}\in \mathbb R$.

For each $\rho\in\Sigma(1)$, let $\eta_{\rho}$ denote the primitive vector in the colattice of $X$.  
The interpretation of curves in terms of relations together with the combinatorial classification of minimal families of rational curves given in \cite{CFH} naturally lead to
consider the following set 
$$
\big\{\small{\textstyle\sum}_{\rho\in \Sigma(1)}\kappa_{\rho}a_{\rho}: {\textstyle\sum}_{\rho\in \Sigma(1)}a_{\rho}\eta_{\rho}=0, ~a_{\rho}\in \mathbb Z_{\geq 0} , ~ \forall~ \rho \in \Sigma(1)\mbox{ and } (a_\rho)_\rho\neq\bf 0%
\big\}
$$
and along with Theorem~\ref{thm:main1} infer the following theorem (see Section~\ref{sec:Proof2} for details).

\begin{theorem}\label{thm:toric}
Let $(X,\omega)$ be a compact K\"ahler toric manifold and let $[\omega]$ be given by Eq.(1.1). Then 
\begin{equation}\label{Eq:toric}
   w_G(X,\omega) \leq \min
\big\{\small{\textstyle\sum}_{\rho\in \Sigma(1)}\kappa_{\rho}a_{\rho}: {\textstyle\sum}_{\rho\in \Sigma(1)}a_{\rho}\eta_{\rho}=0,  ~a_{\rho}\in \mathbb Z_{\geq 0} , ~ \forall~ \rho \in \Sigma(1)\mbox{ and } (a_\rho)_\rho\neq\bf 0 
\big\}.
\end{equation} 
Moreover, the above minimum is attained in case $a_{\rho}\leq 1$ for all $\rho$.
\end{theorem}

The first assertion of Theorem~\ref{thm:toric} extends Lu's theorems obtained in ~\cite{Lu06a} for Fano smooth projective toric varieties and their blow ups at torus fixed points;
both assertions give affirmative answers to some questions raised in \cite{HLS}.
We notice also that one of these questions is stated as a conjecture in~\cite{AHN}.  
All this is discussed in detail in Subsection~\ref{sec:comparison}. 

Section~\ref{sec:Seshadri} on Seshadri constants concludes this work. Gromov widths and Seshadri constants of projective complex manifolds $X$ equipped with a very ample line bundle are closely related: as proved in \cite[Proposition 6.2.1]{BC01}, the latter is upper bounded by the former.
Theorem~\ref{thm:main1} and Theorem~\ref{thm:toric} thus yield upper bounds of Seshadri constants for the varieties under consideration, as stated in Corollary~\ref{cor:Seshadri-general} and Corollary~\ref{cor:Seshadri-toric}.


\section{Upper bounds for the Gromov width of K\"ahler manifolds}\label{sec:upperbound}

The purpose of this section is to prove Theorem~\ref{thm:main1}. 
As already mentioned in the introduction, a key ingredient in our proof is a theorem essentially due to Gromov; the latter is recalled in the first subsection. In the second subsection, we review important notions and results on the curves we consider, that are the minimal curves. We thus proceed with the proof Theorem~\ref{thm:main1}.

\subsection{Gromov's Theorem}\label{sec:GW}

In order to state properly Gromov's Theorem, we shall start with some recalls on Gromov-Witten invariants.


In this subsection, $(X,\omega)$ denotes a symplectic manifold.
 
Given  $A\in H_2(X, \mathbb Z)$ and $J$ an almost complex structure on $X$ compatible with $\omega$, consider the moduli space $\overline{\mathcal M}^X_{k}(A,J):=\overline{\mathcal M^X_{0,k}(A,J)}$ 
of $J$-holomorphic stable maps to $X$ of genus $0$, of class $A$ and with $k$ marked points. 
This space carries a virtual fundamental class $[\overline{\mathcal M}^X_{k}(A)]^{vir}$, independent of $J$, and in
the rational \v{C}ech homology group $\check{H}_{d}(\overline{\mathcal M}^X_{k}(A,J), \mathbb Q)$ where $d$ denotes the expected dimension of
$\overline{\mathcal M}^X_{k}(A,J)$, that is
$$
d = dimX + 2 c_1(A) + 2k-6
$$
with $c_1$ being the first Chern class of the tangent bundle of $X$.

Let
$$
ev^k:\overline{\mathcal{M}}^X_{k}(A,J)\longrightarrow X^k
$$
be the evaluation map sending a stable map to the $k$-tuple of its values at the $k$ marked points and
$$
\pi: \overline{\mathcal{M}}^X_{k}(A,J) \longrightarrow\overline{\mathcal{M}}_k
$$
be the forgetful map with target the moduli of stable curves of genus $0$ with $k$ marked points.

Since the space $X^k \times \overline{\mathcal{M}}_k$ is both paracompact and locally contractible, the singular cohomology and \v{C}ech cohomology are equivalent. Consequently, the push-forward of $[\overline{\mathcal{M}}^X_{k}(A)]^{vir}$ under $ev^k\times \pi$ can be interpreted as a rational singular homology class.

For $\alpha_i\in H^*(X,\mathbb Q)$ with $i= 1,...,k$ and $\beta\in H_*(\overline{\mathcal{M}}_k,\mathbb Q)$, the
Gromov-Witten invariant is defined to be the rational number
$$
GW^{(X,\omega)}_{A,k}(\alpha_1,...,\alpha_k;\beta):=
\int_{(ev^k\times \pi)_{*}([\overline{\mathcal M}^X_{k}(A)]^{vir})} \alpha_1\times ...\times\alpha_k\times 
 \mathrm{PD}\beta,
$$
whenever the degrees of $\alpha_1,\ldots,\alpha_k,\beta$ sum up to the expected dimension $d$; otherwise it is $0$.

As mentioned above, we consider $(ev^k\times \pi)_{*}([\overline{\mathcal M}^X_{k}(A)]^{vir})$ as a rational singular homology class. When defining the Gromov-Witten invariant, the integration refers to the pairing of this class with a cohomology class, which yields a rational number.

The following theorem is well-known; it is thoroughly proved e.g. in~\cite{HLS} by using ideas of Gromov's. Cf.~\cite[Theorem 1.27]{Lu06b}.

\begin{theorem}[Gromov] \label{thm:Gromov}
Let $(X,\omega)$ be a symplectic manifold and
$A \in H_2(X, \mathbb Z)$ be a non-trivial second homology class.
Suppose $GW^{(X,\omega)}_{A, k}(PD[pt], \alpha_2, \ldots, \alpha_k;\beta)\neq 0$ for some $k$, 
$\alpha_i\in H^*(X, \mathbb Q)$ and $\beta\in H_*(\overline{\mathcal{M}}_k,\mathbb Q)$. 
Then the inequality $w_G(X,\omega)\leq \int_A\omega$ holds.
\end{theorem}

To prove this theorem, the authors use the property that the virtual fundamental class lies in the rational homology group $\check{H}_{2d}(\overline{\mathcal M}^X_{k}(A,J), \mathbb Q)$; see~\cite[Remark 4.3]{HLS}.
This property is satisfied e.g. for the virtual fundamental class constructed in~\cite{Castellano}.


\subsection{Minimal curves} Let us now recall some basic notions on  minimal rational curves from \cite[Chapter II.2]{Kol}. 

In this subsection, $X$ denotes a smooth projective complex algebraic variety.

Let $\mathrm{RatCurves}(X)$ denote the normalization of the space of rational curves on $X$. 
Every irreducible component $\mathcal K$ of $\mathrm{RatCurves}(X)$ is a (normal) quasi-projective variety equipped with a quasi-finite morphism to the Chow variety of $X$; the image consists of the Chow points of irreducible, generically reduced rational curves.
Every such $\mathcal K$ is called a family of rational curves on $X$.
There exist a universal family $p :\mathcal U \to \mathcal K$ and a projection $\mu :\mathcal U \to X$.
For any $x\in X$, let $\mathcal U_x=\mu^{-1}(x)$ and $\mathcal K_x=p(\mathcal U_x$).
A family $\mathcal K$ is called a \emph{covering family} if $\mu$ is dominant, i.e., $\mathcal K_x$ is non-empty for a general point $x\in X$. If in addition $\mathcal K_x$ is projective for a general point $x$, then $\mathcal K$ is called a \emph{minimal family}. 

A rational curve $f:\mathbb P^1\to X$ is \emph{free} 
if $H^1(\mathbb P^1, f^*T_X)=0$, where $T_X$ is the tangent bundle of $X$; 
see \cite[Definition II.3.1]{Kol}. 

Recall that a {\it very general point} of $X$ is a point outside a countable union of proper closed subvarieties of $X$.  
The two following theorems provide sufficient conditions for  free rational curves to exist.
\begin{theorem}[{\cite[Theorem II.3.11]{Kol}}]\label{thm:free curve}
Any rational curve passing through a very general point is free.
\end{theorem}

The next theorem can be derived from the proof of Theorem IV.1.9 in loc. cit. and \cite[Proposition 1.1]{KMM}.
Recall that a smooth complex projective variety $X$ is {\it uniruled} if for any point $x\in X$, there is a rational curve passing through $x$.

\begin{theorem}\label{thm:free+universal}
Let $\mathcal K$ be a family of rational curves in $X$. Then, $\mathcal K$ is a covering family if and only if there exists a (free) curve in $\mathcal K$ that passes through a very general point. Furthermore, if such a family exists, it is equivalent to $X$ being uniruled.
\end{theorem}

Following \cite{Hwa14}, we define the notion of minimal-degree covering families. Given an ample line bundle $\mathcal L$ on $X$. 
By $deg_{\mathcal L}(\mathcal K)$, we denote the degree of $\mathcal L$ on one hence all members of $\mathcal K$.
A covering family $\mathcal K$ is called a {\it minimal covering  with respect to $\mathcal L$}, if $deg_{\mathcal L}(\mathcal K)$ is minimal among all covering families of $X$. A covering family $\mathcal K$ is called a {\it minimal-degree covering} if it is minimal with respect to some ample line bundle, and in this case, any member of $\mathcal K$ is called a {\it minimal-degree curve}.

\begin{lemma}\label{lem:md}
The minimum of $deg_{\mathcal L}(\mathcal K)$ over all minimal families $\mathcal K$ of $X$ is attained for a minimal-degree covering family.
\end{lemma}

\begin{proof}
By~\cite[Section 3]{Hwa14}, any minimal-degree covering family is a minimal family.
Moreover, a minimal family is a covering family by definition. The lemma thus follows. 
\end{proof}

\begin{theorem}[{\cite[Theorem IV.2.10]{Kol}}]\label{thm:existence of minimal}
Minimal covering families exist on any uniruled variety.
\end{theorem}


\subsection{Proof of Theorem \ref{thm:main1}}
We first reduce the proof to the case of an integral Kähler class; this mainly follows from the lower semicontinuity of the Gromov-width as a function on Kähler forms. This property is known by the experts; we recall the main idea of its proof for convenience.

\begin{lemma}\label{lemma:semicont}
Let $(M,\omega)$ be a compact symplectic manifold.
The function $\omega\mapsto w_G(M,\omega)$ is lower semicontinuous
on the space of symplectic forms on $M$ equipped with the $\mathcal C^1$-topology.
\end{lemma}

\begin{proof}
Let $a$ be the Gromov width of $(M,\omega)$. Note that $a>0$ by Darboux Theorem.
For each $\epsilon>0$, let $\psi_\epsilon:(B_{a-\epsilon}, \omega_{std})\rightarrow (M, \omega)$ denote a symplectic embedding of a ball of capacity $a-\epsilon$.
Let $\mathcal U$ be a $\mathcal C^1$-neighbourhood of $\omega$
such that the $2$-form $\Omega_{\tau,t,\epsilon}=\psi_\epsilon^*(\omega)+t\psi_\epsilon^*(\tau-\omega)$ of $B_{a-\epsilon}$ 
is symplectic for each $0\leq t\leq 1$ and each $\tau\in \mathcal U$.
Note that $\psi_\epsilon^*(\tau-\omega)$ is exact by Poincar\'e Lemma.
Taking $\epsilon>0$ and $\epsilon'<\epsilon$, we now apply a Moser type argument to each family 
$\mathfrak F_{\tau}=\{\Omega_{\tau,t,\epsilon'}: 0\leq t\leq 1\}$, with $\tau\in\mathcal U$, 
as in the proof of~\cite[Proposition 11.1]{EV} (see also the end of the proof of~\cite[Proposition 14]{MP}). 
We thus get 
a $\mathcal C^1$-neighbourhood $\mathcal{U'}$ of $\omega$ contained in $\mathcal U$ such that for all $\tau\in\mathcal U'$ there exists 
a smooth embedding $\phi_\tau: B_{a-{\epsilon''}}\rightarrow B_{a-\epsilon'}$ with $\phi_\tau^*(\Omega_{\tau,1,\epsilon'})=\Omega_{\tau,0,\epsilon'}(=\omega_{std})$ for some $\epsilon''=\epsilon''(\tau)$ satisfying $\epsilon'<\epsilon''<\epsilon$. 
From this, for each $\tau\in\mathcal U'$ we obtain a symplectic embedding 
$\psi_{\epsilon'} \circ\phi_\tau:(B_{a-\epsilon''},\omega_{std})\rightarrow (M,\tau)$.
This implies the lower semicontinuity of the Gromov width.
\end{proof}

Let $a$ be the Gromov width of $(X,\omega)$.
Then thanks to Lemma~\ref{lemma:semicont} and the projectivity of $X$, for any given $\epsilon>0$, there exists a $\mathcal C^1$-neighbourhood $\mathcal U(\epsilon)$ of $\omega$ and a Kähler form $\tau\in\mathcal U(\epsilon)$ with $[\tau]\in H^2(X,\mathbb Q)$ such that 
the Gromov width of $(X,\tau)$ is greater or equal to $a-\epsilon$. 
Therefore, if the theorem holds for K\"ahler forms with classes in $H^2(X,\mathbb Q)$,
the inequality $a-\varepsilon\leq \int_C\tau$ holds for any minimal curve $C$ and so does the inequality $a\leq \int_C\omega$.
We are thus left to prove the theorem for rational K\"ahler classes $[\omega]$ and even only for integral ones thanks to the conformality of the Gromov width.

Let us now consider the case of Kähler forms $\omega$ with $[\omega]\in H^2(X,\mathbb Z)$.
Thanks to Lemma~\ref{lem:md}, for any minimal curve $C$ of $X$, there exists a minimal-degree curve $C_0$ of $X$ such that $\int_{C_0}\omega\leq \int_{C}\omega$.
Moreover, since any covering (hence minimal) family contains a (free) curve passes through a very general point, as recalled in Theorem~\ref{thm:free+universal},
it suffices to prove Theorem \ref{thm:main1} for minimal-degree curves which are free.
Furthermore, by Theorem~\ref{thm:Gromov}, we are left to prove the following theorem to obtain Theorem \ref{thm:main1}.

\begin{theorem}\label{thm:GW-integral}
Let $(X,\omega)$ be a
projective K\"ahler manifold. Let $C$ be a (free) curve on $X$ that passes through a very general point 
and it is also a minimal-degree curve.
Then
$$
GW^{(X,\omega)}_{[C], k}(PD[pt], \alpha_2, \ldots, \alpha_k;[pt])\neq 0  
$$
for some $k$, $\alpha_i\in H^*(X, \mathbb Q)$ and with the last $[pt]$ being the point class of $[\overline{\mathcal M}_k]$.
\end{theorem}

The above theorem is essentially due to Koll\'ar and Ruan; see the proof of \cite[Theorem 4.2.10]{Kol98} and of \cite[Proposition 4.9]{Ruan}. 
As a sake of convenience and since we deal with different definitions of Gromov-Witten invariants, we outline the proof.


\begin{proof}
First we claim that the curve $C$ is irreducible.
Indeed, if $C'$ is an irreducible component of $C$ then $C'$ is also free 
hence it is contained in a covering family
by Theorem \ref{thm:free+universal}. 
Let $\mathcal L$ be a very ample line bundle with respect to which 
$C$ is a minimal-degree curve.  
As $\mathcal L$ is very ample, it follows that $\mathcal L \cdot C'> 0$ and in turn, $\mathcal L \cdot C' < \mathcal  L \cdot C$ whenever $C\neq C'$. The claim thus follows.

Let $J$ denote the complex structure of the K\"ahler manifold $(X,\omega)$.
Let
$
\overline{\mathcal{M}}_{pt}
$
consist of the equivalences classes of the stable $J$-holomorphic curves $[(u,z_1,\ldots,z_k)]$ in $\overline{\mathcal M}^X_{k}([C],J)$ such that $u(z_1)$ is a given very general point of $X$
and $\mathcal{M}_{pt}\subset\overline{\mathcal{M}}_{pt}$ be 
given by the equivalent classes of $J$-holomorphic spheres in the top stratum.
Note that $\mathcal{M}_{pt}$ is not empty thanks to the assumption on $C$.
Since $\mathcal{M}_{pt}$ consists of free rational curves, it is a smooth manifold and it has the expected dimension, that is 
$d-\dim X$; see e.g. §3.3 in~\cite{MDS}.
Therefore $[\overline{\mathcal M}_{pt}]^{vir}=[\overline{\mathcal M}_{pt}]$. 

Besides, the curve $C$ being irreducible, the class $[C]$ is indecomposable (in the sense of \cite[Definition 7.1.7]{MDS}) hence
the evaluation map $ev_{k-1}:\mathcal M_{pt}\rightarrow X^{k-1}$ at the last $k-1$ points  defines a pseudo-cycle (see~Lemma 7.1.8 in~\cite{MDS}). Therefore
\begin{equation}\label{eq:choicecocycles}
\int_{ev_{k-1*}([\overline{\mathcal M}_{pt}])} \alpha_2\times\ldots\times\alpha_k
=\int_{ev_{k-1*}([\mathcal M_{pt}])} \alpha_2\times\ldots\times\alpha_k
\end{equation}
where $\alpha_2,\ldots,\alpha_k$ are cohomology classes of $X$.
We further choose these classes such that
\begin{equation}\label{eq:choicecocycles2}
\int_{ev_{k-1*}([\mathcal M_{pt}])} \alpha_2\times\ldots\times\alpha_k\neq 0.
\end{equation}
One may take for instance powers of the integral K\"ahler form $\omega$ for the cocycles $\alpha_2,\ldots,\alpha_k$ such that they satisfy the dimension condition. 

Finally, note that the following equalities hold:
\begin{equation*}
\begin{split}
GW^{(X,\omega)}_{[C], k}(PD[pt],\alpha_2, \ldots,\alpha_k;[pt])
& =\int_{ev^{k}_*([\overline{\mathcal M}^X_{k}([C])]^{vir})} PD[pt]\times\alpha_2\times\ldots\times \alpha_k \\
& =\int_{ev_{k-1}*([\overline{\mathcal M}_{pt}]^{vir})} \alpha_2\times\ldots\times\alpha_k.
\end{split}
\end{equation*}
This combined with Equations (\ref{eq:choicecocycles}) and (\ref{eq:choicecocycles2}) yields 
$GW^{(X,\omega)}_{[C], k}(PD[pt],\alpha_2, \ldots,\alpha_k;[pt])\neq 0$.
\end{proof}

\begin{remark}
As shown in the above proof, the Gromov-Witten invariant considered in Theorem~\ref{thm:GW-integral} can be expressed via pseudo-cycles
only. Therefore, in order to derive Theorem~\ref{thm:main1}, we could have applied a theorem of Castro's and more precisely the arguing of the proof of Theorem 4.5 together with Remark 4.6 in ~\cite{Cas}.
\end{remark}


\section{Upper bounds for the Gromov width of toric varieties}
\label{sec:Proof2}

The main goal of this section is to prove Theorem~\ref{thm:toric}.
We start by recalling a few basic results on toric varieties as well as the classification of minimal rational curves on toric manifolds (see Theorem \ref{thm:Fu}).
We conclude this section while comparing our results with ones previously obtained.

\subsection{Free curves on toric varieties} 
Let us recall some notions on toric varieties. We refer to \cite{cox} for more details.
 
 Let $N$ be the lattice of one-parameter subgroups of the torus $T=(\mathbb C^*)^n$ and let $M$ be 
 the lattice of characters of $T$. Set $M_{\mathbb R}:=M\otimes_{\mathbb Z} \mathbb R$ and 
 $N_{\mathbb R}:=N\otimes_{\mathbb Z} \mathbb R$. We have a natural bilinear pairing 
  $$\langle-,-\rangle:M_{\mathbb R}\times N_{\mathbb R}\to \mathbb R.$$
  
Given a fan $\Sigma$ in $N_{\mathbb R}$, let $\Sigma(1)$ be the set of one-dimensional cones in the fan $\Sigma$.
To such a fan $\Sigma$, we can associate a toric $T$-variety that we denote by $X$ below. For each $\rho\in \Sigma(1)$, let $D_{\rho}$ be the associated $T$-invariant prime divisor in $X$.
The group $TDiv(X)$ of  $T$-invariant divisors in $X$ is given by 
\begin{equation}\label{eq:divisor}
TDiv(X)=\bigoplus_{\rho\in\Sigma(1)}\mathbb ZD_{\rho}.
\end{equation}
Let $Cl(X)$ be the divisor class group of $X$.
We have an exact sequence 
$$
M\overset{\alpha}{\longrightarrow}\mathbb Z^{|\Sigma(1)|}\overset{\beta}{\longrightarrow} Cl(X)\to 0
$$
where $\alpha(m)=(\langle m,\eta_{\rho}\rangle)_{\rho\in \Sigma(1)}$ and $\beta$ maps the standard basis element $e_{\rho}\in \mathbb Z^{|\Sigma(1)|}$ to $[D_{\rho}]\in Cl(X)$; see \cite[Theorem 4.1.3]{cox}.

By $N_1(X)$ we denote the group of numerical classes of 1-cycles of the variety $X$ and by $Pic(X)$ the Picard group of $X$. Let $N_1(X)_{\mathbb R}:=N_1(X) \otimes_{\mathbb Z} \mathbb R$ and $Pic(X)_{\mathbb R}=Pic(X)\otimes_{\mathbb Z}\mathbb R$.
\begin{proposition}(\cite[Proposition 6.4.1]{cox})\label{ses} Let $X$ be a smooth complete toric $T$-variety. The following sequence
$$
 0\longrightarrow M_{\mathbb R} \overset{\alpha}{
\longrightarrow} \mathbb R^{|\Sigma(1)|} \overset{\beta}{\longrightarrow} Pic(X)_{\mathbb R}\longrightarrow 0
$$  
is exact and so is its dual
$$ 
0\longrightarrow N_1(X)_{\mathbb R} \overset{\beta^*}{\longrightarrow}  \mathbb R^{|\Sigma(1)|} \overset{\alpha^*}{\longrightarrow} N_{\mathbb R} \longrightarrow 0.
$$

Furthermore, given $D=\sum \limits_{\rho\in \Sigma(1)} \kappa_{\rho}D_{\rho}$ and a relation $\sum \limits_{\rho \in \Sigma(1)} a_{\rho}\eta_{\rho}=0$, the intersection pairing of $[D]\in Pic(X)_{\mathbb R}$ and $R=(a_{\rho})_{\rho\in \Sigma(1)}\in N_1(X)_{\mathbb R}$ is 
\begin{equation}\label{eq:intersection}
    D \cdot R=\sum \limits_{\rho\in \Sigma(1)} \kappa_{\rho}a_{\rho}.   
    \end{equation}
\end{proposition}

The above proposition shows in particular that $N_1(X)_{\mathbb R}$ can be interpreted as the space of linear relations among the $\eta_p$'s.

\begin{lemma}\label{lemma:curve} Let $X$ be a smooth complete toric $T$-variety.
To any relation $\sum_{\rho\in \Sigma(1)} a_{\rho}\eta_{\rho}= 0$ with $(a_{\rho})_{\rho\in \Sigma(1)}\in\mathbb Z^{|\Sigma(1)|}_{\geq 0}$ non equal to $\bf 0$, there corresponds a free irreducible rational curve of $X$ passing through a very general point. 
\end{lemma}

This lemma is essentially \cite[Proposition 2]{Sam} although there is no mention of free curve therein. As a sake of convenience, we give the proof below.

\begin{proof} 
Take a relation $\sum_{\rho\in \Sigma(1)} a_{\rho}\eta_{\rho}= 0$.
We claim that there exists a curve $C$ of $X$ such that 
\begin{equation}\label{eq:Payne}
D_{\rho}\cdot C = a_{\rho} \quad \text{for all} \quad \rho\in \Sigma(1).
\end{equation}
This curve $C$ is constructed as follows by mimicking the proof of~\cite[Proposition 2]{Sam}.
Given $\rho\in \Sigma(1)$, let $\lambda_{\eta_{\rho}}: \mathbb C^*\to T$ be the one-parameter subgroup of $T$ associated to the primitive vector $\eta_{\rho}$. Namely, $\lambda_{\eta_{\rho}}(t)=(t^{{\eta_{\rho}}_1},\ldots, t^{{\eta_{\rho}}_n})$ where $\eta_{\rho}:=({\eta_{\rho}}_1, \ldots, {\eta_{\rho}}_n)$. 
For each $\rho\in \Sigma(1)$, take $c_{\rho}\in \mathbb C$ such that the scalars $c_{\rho}$ are all distinct. 

Consider the rational map $\tilde f: \mathbb A^1\dashrightarrow T$ defined by 
\begin{equation}\label{eq:ratmap}
     \tilde  f(t)=\prod_{\rho\in \Sigma(1)}\lambda_{\eta_{\rho}}(t-c_{\rho})^{a_{\rho}}.
 \end{equation}
Since $X$ is complete, the map $\tilde f$ extends to a regular morphism $f:\mathbb P^1 \to X$. 

Let $C:= f(\mathbb P^1)$. We shall now prove that $C$ is the required curve.  First, it is clear that $C$ is an irreducible rational curve in $X$. Moreover, note that since $\eta_{\rho}$ is the primitive vector of the ray $\rho$, we have $\lim_{t\to 0}\lambda_{\eta_{\rho}}(t)\in D_{\rho}$ (see \cite[Proposition 3.2.2]{cox}). Using Equation (\ref{eq:ratmap}) together with the equality $X\setminus T=\cup_{\rho\in \Sigma(1)} D_{\rho}$, we obtain the following assertions:
\begin{itemize}
\item  If $a_{\rho}=0$, then $C \cap D_{\rho}=\emptyset$.
\item  If $a_{\rho}>0$, then $C\cap D_{\rho}= f(c_{\rho})$ with $a_{\rho}$ being the multiplicity at $f(c_{\rho})$. 
\end{itemize}

By construction, the curve $C$ passes through a very general point of $X$.
Therefore, by Theorem \ref{thm:free curve}, the rational curve $C$ is free.
Moreover, $C$ corresponds to the given relation by Equations~(\ref{eq:intersection}) and~(\ref{eq:Payne}). The curve $C$ itself thus yields the desired curve.
\end{proof}

\subsection{Minimal curves on toric varieties}
We now recall the combinatorial description of  minimal rational curves on smooth complete toric varieties obtained in~\cite{CFH}. 
We shall need the following notation attached to a complete fan $\Sigma$. 

\begin{definition}[{\cite{Bat}}] \label{def:primcoll}
A non-empty subset $\mathfrak P=\{x_1, \ldots, x_k\}$ of $\Sigma(1)$ is called a \emph{primitive collection} if, 
for any $1\leq i\leq k$, the set $\mathfrak P\setminus \{x_i\}$ generates a $(k-1)-$dimensional cone in $\Sigma$, but $\mathfrak P$ does not generate a $k$-dimensional cone in $\Sigma$.  
\end{definition}
For a primitive collection $\mathfrak P=\{x_1, \ldots, x_k\}$  of $\Sigma(1)$, let $\sigma(\mathfrak P)$ be the unique cone in $\Sigma$ that contains $x_1+\cdots +x_k$ in its interior. 
Let $y_1, \ldots, y_m$ be generators of $\sigma(\mathfrak P)$. Then,
there exists a unique equation  such that 
$$
x_1+\cdots+x_k=b_1y_1+\ldots+b_my_m \quad\mbox{ with }\quad b_i\in \mathbb Z_{>0}.
$$
The equation $x_1+\cdots+x_k-b_1y_1-\ldots-b_my_m=0$ is called the \emph{primitive relation} of $\mathfrak P$. 
The \emph{degree} of $\mathfrak P$ is defined as 
$$
deg(\mathfrak P)=k-\sum_{i=1}^{m}b_i.
$$
\begin{theorem}[{\cite[Proposition 3.2 and Corollary 3.3]{CFH}}]\label{thm:Fu} 
Let $X$ be a smooth projective toric $T$-variety.
\begin{enumerate}
      \item There is a bijection between minimal rational components of degree $k$ on $X$ and primitive collections $\mathfrak P=\{x_1, \ldots, x_k\}$ of $\Sigma(1)$ such that $x_1+\cdots+x_k=0$.
      \item There exists a minimal rational component in $\mathrm{RatCurves}(X)$.
\end{enumerate}
\end{theorem}

\subsection{Proof of Theorem \ref{thm:toric}} 

\begin{proof}
Recall Eq.~(\ref{eq:kaehlerclass}), that is, $[\omega]$ is a K\"ahler class on $X$ and $[\omega]=\sum_\rho \kappa_\rho[D_\rho]$ for some $\kappa_\rho\geq 0$.
Let $\sum_{\rho\in \Sigma(1)} a_{\rho}\eta_{\rho}= 0$ with $(a_{\rho})_{\rho\in \Sigma(1)} \in\mathbb Z^{|\Sigma(1)|}_{\geq 0}$ non equal to $\bf 0$. 
We first prove that 
\begin{equation}\label{eq:mainthm}
    w_G(X, \omega) \leq \sum_{\rho\in \Sigma(1)}a_{\rho}\kappa_{\rho}.
\end{equation}
By Lemma \ref{lemma:curve} together with Equation (\ref{eq:Payne}), there is a free rational curve $C$ in $X$ such that 
$$
\int_C {\omega}=\sum_{\rho\in \Sigma(1)} a_{\rho}\kappa_{\rho}.
$$ 
By Theorem \ref{thm:free+universal}, there is a covering family containing $C$. 
Then, there is a minimal-degree curve $C_0$ of $X$ such that $\int_{C_0} {\omega}\leq \int_C {\omega}$.
Thus, the minimum 
$$
\min
\big\{\small{\textstyle\sum}_{\rho\in \Sigma(1)}\kappa_{\rho}a_{\rho}: {\textstyle\sum}_{\rho\in \Sigma(1)}a_{\rho}\eta_{\rho}=0, ~a_{\rho}\in \mathbb Z_{\geq 0} , ~ \forall~ \rho \in \Sigma(1)\mbox{ and } (a_\rho)_\rho\neq\bf 0 
\big\}
$$
is greater than equal to $\int_{C_0} {\omega}$.
By applying Theorem \ref{thm:main1}, we get the desired inequality, that is Equation~(\ref{eq:mainthm}) and in turn, the first assertion of Theorem~\ref{thm:toric}.

Besides, the class of any minimal rational curve corresponds to a primitive collection $\mathfrak P\subset \Sigma(1)$ with the relation $\sum_{\rho\in \mathfrak P} \eta_{\rho}=0$, as stated in Theorem \ref{thm:Fu}. The second assertion of Theorem~\ref{thm:toric} thus follows.
\end{proof}

\subsection{Comparison with previous results}
\label{sec:comparison}
In this section, we compare our results on toric manifolds with some results previously obtained by Lu in \cite{Lu06a}. 
We also address some questions raised in \cite{HLS}.
We conclude this subsection by reformulating Theorem ~\ref{thm:toric} in terms of widths of polytopes and give an affirmative answer to  a conjecture stated in~\cite{AHN}.


\subsubsection{Lu's results in~\cite{Lu06a}}

Keep the notation $(X, \omega)$ as in Theorem~\ref{thm:toric}. 
Set 
$$
\gamma(X,\omega):=\min
\big\{\small{\textstyle\sum}_{\rho\in \Sigma(1)}\kappa_{\rho}a_{\rho}: {\textstyle\sum}_{\rho\in \Sigma(1)}a_{\rho}\eta_{\rho}=0, ~a_{\rho}\in \mathbb Z_{\geq 0} , ~ \forall~ \rho \in \Sigma(1)\mbox{ and } (a_\rho)_\rho\neq\bf 0 
\big\}
$$
and 
    $$\Lambda(X,\omega):=\max\{{\textstyle\sum}_{\rho\in \Sigma(1)}\kappa_{\rho}a_{\rho}: (a_{\rho})_{\rho\in \Sigma(1)}\in S \},$$
where  $$S:=\big \{ (a_{\rho})_{\rho\in \Sigma(1)}\in \mathbb Z^{|\Sigma(1)|}_{\geq 0}:  {\textstyle\sum}_{\rho\in \Sigma(1)}a_{\rho}\eta_{\rho}=0 ~\text{and}~ 1\leq {\textstyle\sum}_{\rho\in\Sigma(1)}a_{\rho}\leq n+1\big \}.$$

In  \cite[Theorem 1.2]{Lu06a} and \cite[Theorem 6.2]{Lu06a}, Lu obtained the upper bound $\gamma(X,\omega)$ 
for Fano smooth projective toric varieties and their blow ups at torus fixed points resp.; his proof makes use of Batyrev’s computations of the quantum cohomology of these varieties. Theorem~\ref{thm:toric} thus extends these results of Lu's to any smooth projective toric variety.

Besides, for any smooth projective toric variety, Lu proved that the Gromov width of a compact K\"ahler toric manifold 
is bounded from above by  $\Lambda(X,\omega)$; see \cite[Theorem 1.1]{Lu06a}. 
By Lemma~\ref{lem:bounds} and Example \ref{example} below, Theorem~\ref{thm:toric} gives a sharper bound than loc. cit.\textbf{}

\begin{lemma}\label{lem:bounds} The inequality $\gamma(X,\omega)\leq \Lambda(X,\omega)$ holds. 
\end{lemma}

\begin{proof} 
By the definition of primitive collections and Theorem \ref{thm:Fu}, it is clear that the primitive collection $\mathfrak P=\{x_1, \ldots, x_k\}$ corresponding to a minimal family satisfies the conditions in $S$. Namely, in this case, the equation ${\textstyle\sum}_{\rho\in \Sigma(1)}a_{\rho}\eta_{\rho}=0$ reads as  $\sum_{i=1}^k 1 x_i=0$ and we have $1\leq \sum_{i=1}^k 1\leq n+1$. 
The proof follows.
\end{proof}

\begin{example}\label{example}
Consider the Hirzebruch surface  $\mathscr H_2=\mathbb P(\mathcal O_{\mathbb P^1}\oplus \mathcal O_{\mathbb P^1}(2))$. The fan of $\mathscr H_2$ in $\mathbb R^2$ is given by the ray generators 
$$
u_1=(-1, 2),\quad u_2=(0, 1),\quad u_3=(1, 0) \quad\text{and} \quad u_4=(0, -1).
$$
The primitive relations of $\mathscr H_2$ are given by 
$$ u_2+u_4=0\quad \text{and}\quad u_1+u_3=2u_2$$
with degrees $2$ and $0$ respectively.

Note that a complete smooth toric variety $X$ is Fano if and only if $\deg(\mathfrak P)>0$ for all primitive collections $\mathfrak P$ of $X$; see \cite[Proposition 2.3.6]{Bat2}. By applying this characterization, we can show that $\mathscr H_2$ is not Fano.  

Besides, $\mathscr H_2$ is not a blowup of a toric surface at torus fixed points. 
Indeed, any smooth complete toric surface is obtained by a finite sequence of blowups at torus fixed points from either $\mathbb P^2$, $\mathbb P^1 \times \mathbb P^1$, or one of the Hirzebruch surfaces $\mathscr H_r$ with $r\geq 2$; see \cite[Theorem 10.4.3]{cox}.

Finally, note that the Picard group of $\mathscr H_2$ is generated by $D_{3}$ and $D_4$. Furthermore, a divisor $D=aD_3+bD_4$ is (very) ample if and only if $a, b >0$; see \cite[Page 273, Eq. 6.1.12]{cox}. 
Take the K\"ahler form $\omega$ associated to a very ample divisor $D=aD_3+bD_4$ with $a, b>0$.

Since the following relations  hold:
$u_2+u_4=0 \quad \text{and}\quad u_1+u_3+2u_4=0$, we have the inequalities
$$
\Lambda(\mathscr H_2,\omega)\geq a+2b\quad \text{and}\quad \gamma(\mathscr H_2,\omega)\leq b.
$$
And in turn, we have $\gamma(\mathscr H_2,\omega)<\Lambda(\mathscr H_2,\omega)$. 
\end{example}

\begin{remark}
The bound given in Theorem \ref{thm:toric} may not be sharp if $X$ is not Fano as shown in \cite[Example 5.6]{HLS}. Besides, even if $X$ is a Fano symplectic toric manifold, it is not known whether its Gromov-width is the symplectic area of a minimal curve; see \cite[Example 5.7 and Question 5.9]{HLS}. 
\end{remark}


\subsubsection{Some questions raised in~\cite{HLS}}

We notice that the first (resp. second) assertion of Theorem \ref{thm:toric} is stated as a question in~\cite[Question 5.10]{HLS}
(resp.~\cite[Question 5.11]{HLS}).


\subsubsection{Conjecture 3.12 in~\cite{AHN}}
We conclude this section by noticing that Theorem~\ref{thm:toric} can be formulated by means of the lattice width of the momentum polytope $P(X,\omega)$ of $(X, \omega)$.

Recall that the lattice width of a convex polytope $P$ in $\mathbb R^n$ is defined as follows. 
First, the width of $P$ with respect to a non-zero linear functional ${\bf u}\in \mathrm{Hom}(\mathbb R^n, \mathbb R)$ is 
$$
\width_{\bf u}(P):=\max_{x, y\in P}|{\bf u}(x)-{\bf u}(y)|,
$$
and the \textit{lattice width} of $P$ is defined as 
$$
\width(P):=\min_{\mathbf{u}}\width_{\bf u}(P),
$$
where $\mathbf u$ runs over $\mathrm{Hom}(\mathbb Z^n, \mathbb Z)$.  
As proved in~\cite[Proposition 3.9]{AHN}, the following equality holds
\begin{equation}\label{eq:width}
    \gamma(X,\omega)= \width(P(X,\omega)).
\end{equation}

Thanks to this equality, Theorem \ref{thm:toric} thus reads

\begin{corollary}\label{cor:width}  
The Gromov width of a compact K\"ahler toric manifold is bounded from above by the lattice width of its momentum polytope.
\end{corollary}

The above result is stated as a conjecture in~\cite[Conjecture 3.12]{AHN}.


\section{Seshadri constants}\label{sec:Seshadri}

In this section, we give upper bounds for the Seshadri constants of uniruled smooth projective complex varieties. Under certain conditions for toric manifolds, we show that Gromov widths equal Seshadri constants.

First recall Demailly's definition of the {\it Seshadri constant} $\varepsilon(X, \mathcal L, x)$ of a line bundle $\mathcal L$ on a smooth projective complex
variety $X$ at a point $x\in X$:
$$\varepsilon(X, \mathcal L, x):=\inf_{ C}\frac{\mathcal L\cdot C}{mult_xC},
$$
the infimum being taken over all reduced irreducible curves $C\subset X$ passing through $x$ 
and $mult_xC$ being the multiplicity of $C$ at $x$.

As the following proposition shows, Gromov widths and Seshadri constants of smooth projective varieties are closely related.
\begin{proposition}[{\cite[Proposition 6.2.1]{BC01}}]\label{prop:GSrelation} 
Let $X$ be a smooth projective complex algebraic variety equipped  with a very ample line bundle $\mathcal L$. For any point $x\in X$, the inequality
 $$
 \varepsilon(X, \mathcal L, x)\leq w_G(X, \omega_{\mathcal L})
 $$
holds, with $\omega_\mathcal L$ being the Fubini-Study form associated to $\mathcal L$.
\end{proposition}

We derive the following statement from this proposition along with Theorem~\ref{thm:main1}.

\begin{corollary}\label{cor:Seshadri-general}
Let $X$ be a uniruled smooth projective complex algebraic variety equipped with a very ample line bundle $\mathcal L$. 
For any point $x\in X$, the following inequality holds
$$
\varepsilon(X, \mathcal L, x)\leq   \mathrm{min}\big\{\mathcal L\cdot C : \mbox{$C$ minimal curve of $X$}\big\}.
$$
\end{corollary}

The Seshadri constants of line bundles on toric varieties at torus fixed points are computed in ~\cite{Rocco}.
Theorem \ref{thm:toric} yields an estimate for Seshadri constants on toric varieties at any point:

\begin{corollary}\label{cor:Seshadri-toric}
Let $X$ be a smooth projective complex toric variety equipped with a very ample line bundle $\mathcal L$.
For any point $x\in X$, the following inequality holds
$$
\varepsilon(X, \mathcal L, x)\leq \gamma(X,\omega_\mathcal L).
$$
\end{corollary}

We now consider the case where the Gromov width $w_G(X,\omega_{\mathcal L})$ equals $\gamma(X, \omega_\mathcal L)$ for polarized toric manifolds $(X, \mathcal L)$.
For this, we invoke a result of Ito's (\cite{ito2}), 
giving an estimate of the Seshadri constants of toric varieties at very general points. We recall it below after setting some further notation.

Let $\pi_{\mathbb Z}:M \to \mathbb Z$ be a surjective group homomorphism and let $\pi:M_{\mathbb R}\to \mathbb R$ be the linear projection induced by $\pi_{\mathbb Z}$.
Take the momentum polytope $P$ of $(X, \mathcal L)$ and fix a  $t\in \pi(P)\cap \mathbb Q$ such that $P(t):=P\cap \pi^{-1}(t)$ 
is a $(n-1)$-dimensional polytope. 
Let $X_{P(t)}$ denote the toric variety equipped with the very ample line bundle $\mathcal L_{P(t)}$ 
associated with the polytope $P(t)$. 
Let $|\pi(P)|$ denotes the lattice length of $\pi(P)$ with respect to the lattice $\pi_{\mathbb Z}(M)$. 
Denote the neutral element of the torus $T$ by $1_P\in X$. 

\begin{theorem}[{\cite[Theorem 3.6]{ito2}}]\label{thm:Ito}
Let $(X,\mathcal L)$ be a polarized smooth toric variety and $P$ be its momentum polytope. 
Then
$$
\min\big \{|\pi(P)|,  \varepsilon (X_{P(t)}, \mathcal L_{P(t)}, 1_{P(t)})\big \}\leq  \varepsilon (X, \mathcal L, 1_{P})\leq |\pi(P)|.
$$
\end{theorem}

\begin{corollary} 
Let $(X,\mathcal L)$ be a polarized smooth toric variety and $P$ be its momentum polytope. 
Then the following inequalities hold
$$
\min\big \{|\pi(P)|,  \varepsilon (X_{P(t)}, \mathcal L_{P(t)}, 1_{P(t)})\big \}\leq
\varepsilon (X, \mathcal L, 1_P)
\leq 
w_G(X, \omega_{\mathcal L})\leq\gamma(X, \omega_{\mathcal L})\leq|\pi(P)|.
$$
In particular, we have equalities
if and only if
$
|\pi(P)|\leq \varepsilon (X_{P(t)}, \mathcal L_{P(t)}, 1_{P(t)})   
$
for some $t\in \pi(P)\cap \mathbb Q$.
\end{corollary}

\begin{proof} By Proposition \ref{prop:GSrelation}, the inequality $\varepsilon (X, \mathcal L, 1_{P})\leq w_G(X, \omega_{\mathcal L})$ holds. 
Thanks to Theorem \ref{thm:toric}, we have $w_G(X, \omega_{\mathcal L})\leq\gamma(X, \omega_{\mathcal L})$. 
Equation (\ref{eq:width}) along with definition of the lattice width yield the inequality
$\gamma(X,\omega_\mathcal L)=\width(P) \leq |\pi(P)|$ 
and, in turn, the first assertion.
The last assertion is clear.
\end{proof}

\noindent {\bf Acknowledgments.} We thank the anonymous referee for her/his careful reading of the manuscript and for providing suggestions to improve it.



\begin{thebibliography}{XXXXX}



\bibitem[AHN21]{AHN} G.~Averkov, J.~Hofscheier and B.~Nill,  Generalized flatness constants, spanning lattice polytopes, and the Gromov width, \emph{Manuscripta Math.} \textbf{170} (2023), 147--165.

\bibitem[Bat91]{Bat}
V.~V.~Batyrev, On the classification of smooth projective toric varieties, \emph{T{\^o}hoku Math. J.} \textbf{43(4)} (1991), 569--585.

\bibitem[Bat99]{Bat2} V.~V.~Batyrev, On the classification of toric Fano 4-folds, \emph{J. Math. Sci.} \textbf{94} (1999), 1021--1050. 

\bibitem[Bir01]{biran} P. Biran, From symplectic packing to algebraic geometry and back, \emph{European Congress of Mathematics}, Vol. II (Barcelona, 2000) Progr. Math. \textbf{202}, Birkh\"auser, Basel (2001), 507--524. 


\bibitem[BC01]{BC01} P.~Biran and K.~Cieliebak, Symplectic topology on subcritical manifolds, \emph{Comm. Math. Helv.} \textbf{76} (2001), 712--753.

\bibitem[BCF22]{BCF} 
N.~C.~Bonala and S.~Cupit-Foutou, The Gromov width of Bott-Samelson varieties, \emph{Transform. Groups} (2022).  https://doi.org/10.1007/s00031-022-09765-1.

\bibitem[Can01]{Cannas} A. Cannas da Silva, \emph{Lectures on symplectic geometry}, Lecture Notes in Math., 1764, Springer, Berlin, 2001

\bibitem[Cas16]{Castellano} R. Castellano,  Genus zero Gromov-Witten axioms via Kuranishi atlases, arXiv:1601.04048, 2016 cf. MR3503579.

\bibitem[Cast16]{Cas} A.~C.~Castro, Upper bound for the Gromov width of coadjoint orbits of compact Lie groups, \emph{J. Lie Theory} \textbf{26} (2016), no. 3, 821--860.

\bibitem[CFH14]{CFH}
Y.~Chen, B.~Fu, and J.~M.~Hwang, Minimal rational curves on complete toric manifolds and applications, \emph{Proc. Edinb. Math. Soc.} \textbf{57(2)} (2014), no. 1, 111--123.

\bibitem[CLS11]{cox}
D.~A.~Cox, J.~B.~Little, and H.~K.~Schenck, \emph{Toric varieties}, Graduate studies in mathematics, Vol. 124, American Mathematical Society, Providence, RI, 2011.


\bibitem[DiR99]{Rocco}
S. Di Rocco, Generation of $k$-jets on toric varieties, \emph{Math. Z.} \textbf{231} (1999), 169--188.

\bibitem[EV23]{EV} M. Entov and M. Verbitsky, K\"ahler-type Embeddings of Balls into Symplectic Manifolds, \emph{J. Assoc. Math. Res.} \textbf{1} (2023), no.1, 16--119.

\bibitem[FLP18]{FLP}
X. Fang, P. Littelmann and M. Pabiniak, Simplices in Newton-Okounkov bodies and the Gromov width of coadjoint orbits, \emph{Bull. London Math. Soc.} \textbf{50} (2018), no. 2, 202--218.

  
\bibitem[Gro85]{Gro} M. Gromov, Pseudoholomorphic curves in symplectic manifolds, 
\emph{Invent. Math.} \textbf{82} (1985), no. 2, 307--347.
  
 \bibitem[Hwa14]{Hwa14} J.~M.~Hwang, Mori geometry meets Cartan geometry: varieties of minimal rational tangents, in \emph{Proceedings of the International Congress of Mathematicians}, Seoul 2014, vol. I (Kyung Moon SA, Seoul, 2014), 369–394.
 
\bibitem[HLS21]{HLS}
T.~Hwang, E.~Lee and D.~ Y.~ Suh, 
The Gromov width of generalized Bott manifolds, 
\emph{Int. Math. Res. Not. IMRN} (2021), no.9, 7096--7131.


\bibitem[Ito14]{ito2}  A. Ito, Seshadri constants via toric degenerations, \emph{J. Reine Angew. Math.} \textbf{695} (2014), 151--174.


\bibitem[KT05]{kartol} Y. Karshon and S. Tolman, The Gromov width of complex Grassmannians, \emph{Algebr. Geom. Topol.} \textbf{5} (2005), 911--922.

\bibitem[KMM92]{KMM} J.~ Koll\'ar, Y.~Miyaoka and S.~Mori, Rational connectedness and boundedness of Fano manifolds, \emph{J. Differential Geom.} \textbf{36(3)}  (1992), 765--779.

\bibitem[Kol98]{Kol98} 
J.~ Koll\'ar, Low degree polynomial equations: arithmetic, geometry, topology,  \emph{European Congress of Mathematics}, Vol. I (Budapest, 1996), 255--288, \emph{Progr. Math.}, {\bf 168}, Birkhäuser, Basel, 1998. 


\bibitem[Kol99]{Kol} 
J.~ Koll\'ar, \emph{Rational curves on algebraic varieties}, Vol. 32, Springer Science \& Business Media, 1999.


\bibitem[LMZ15]{LMZ} A. Loi, R. Mossa and F. Zuddas, Symplectic capacities of Hermitian symmetric spaces, \emph{J. Symplectic Geom.} \textbf{13(4)} (2015), 1049--1073.

\bibitem[Lu06a]{Lu06a}G.~Lu,  Symplectic Capacities of Toric Manifolds and Related Results, \emph{Nagoya Math. J.} \textbf{181} (2006), 149–184. 

\bibitem[Lu06b]{Lu06b}G.~Lu, Gromov-Witten invariants and pseudo symplectic capacities, \emph{Israel J. Math.} \textbf{156} (2006), 1--63. 

\bibitem[MP18]{MP}A.~Mandini and M.~Pabiniak, On the Gromov width of polygon spaces, \emph{Transform. Groups} \textbf{23} (2018), 149--183.

\bibitem[MDS12]{MDS} D. McDuff and D. Salamon, \emph{J-holomorphic curves and symplectic topology}, Second edition, American Mathematical Society Colloquium Publications, Vol. \textbf{52} American Mathematical Society, Providence, RI, 2012.

 \bibitem[Pay06]{Sam} S.~Payne, Stable base loci, movable curves, and small modifications, for toric varieties, \emph{Math. Z.} \textbf{253} (2006), 421--431. 


\bibitem[Ru96]{Ruan}Y.~Ruan, Virtual neighborhoods and pseudo-holomorphic curves, \emph{Turkish J. Math.} {\bf 23} (1999), no. 1, 161--231.

\end{thebibliography}
\end{document}